\documentclass{article}

\usepackage{amsmath} 
\usepackage{amsfonts}
\usepackage{amssymb,amscd,amsthm}
\usepackage{geometry}
\usepackage{mathrsfs}
\usepackage{graphicx}
\usepackage{ragged2e}

\usepackage{amsthm} 
\usepackage{enumitem} 
\usepackage{mathtools}
\usepackage{extarrows}
\usepackage{setspace}
\usepackage{tikz}
\usepackage{subfigure}
\usepackage[english]{babel}
\usepackage{graphicx}

\usepackage[center]{caption2}
\usepackage{amsfonts}
\usepackage{mathtools}
\usepackage{multirow}
\usepackage{enumitem}
\usepackage{latexsym}

\newcommand{\RNum}[1]{\expandafter{\romannumeral #1\relax}}

\numberwithin{equation}{section}
\newcommand\asertion[1]{ssertion $({\mathrm{\romannumeral #1\relax}})$}

\newcommand{\asselm}{associative element}

\newcommand{\type}{type}

\newcommand{\bfs}{Without loss of generality we can assume}

\newcommand{\spr}{\mathbb{R}}
\newcommand{\spo}{\mathbb{O}}
\newcommand\huaa[1]{\mathscr{A}(#1)}
\newcommand\hua[3]{\mathscr{#1}^{#2}(#3)}

\newcommand\huac[1]{\mathscr{C}(#1)}
\newcommand\huacc[2]{\mathscr{C}^{#1}(#2)}

\newcommand\conjgt[1]{\overline{#1}}
\newcommand{\re}{\mathit{Re}\,}%
\newcommand\generat[1]{\left\langle #1\right\rangle}%
\newcommand{\Hom}{\text{Hom}}

\def\O{\mathbb{O}}
\def\R{\mathbb{R}}
\def\A{\mathbb{A}}
\def\C{C\ell_7}
\def\abs#1{\left|#1\right|}
\def\P{\mathcal{P}(n)}
\newcommand\clifd[1]{C\ell_{#1}}

\newtheorem{mydef}{Definition}[section]
\newtheorem{rem}[mydef]{Remark}
\newtheorem{eg}[mydef]{Example}
\newtheorem{cor}[mydef]{Corollary}
\newtheorem{prop}[mydef]{Proposition}
\newtheorem{lemma}[mydef]{Lemma}
\newtheorem{thm}[mydef]{Theorem}

\begin{document}

\title{Classification of left octonion modules} 
\author{Qinghai Huo, Yong Li, Guangbin Ren} 
\date{} 
\maketitle 

\begin{abstract}
It is natural to study   octonion Hilbert spaces  as the 
   recently swift development of the theory of quaternion Hilbert spaces. In order to do this, it is important to study first  its algebraic structure,  namely,  octonion modules. In this article, we provide complete classification of left octonion modules. 
   In contrast to the quaternionic setting, we encounter some new phenomena. That is, a submodule generated by one element $m$ may be the whole  module and may be not in the form $\O m$. This motivates us to introduce some new notions such as
   associative elements, conjugate associative elements, cyclic elements. We can characterize octonion modules in terms of these notions. 
   It turns out  that octonions  admit  two distinct structures of octonion modules,  and moreover, the direct sum of their   several copies exhaust all
    octonion modules with finite dimensions.

\end{abstract}

\tableofcontents 

\section{introduction}
The theory of quaternion Hilbert spaces  brings the classical theory of functional analysis into the non-commutative realm (see \cite{horwitz1993QHilbertmod,razon1992Uniqueness,razon1991projection,soffer1983quaternion,viswanath1971normal}).
 It arises some new notions such as   spherical spectrum,  which has potential applications in quantum mechanics  (see \cite{colombo2011noncomfunctcalculus,ghiloni2013slicefct}). All  these theories are based on   quaternion vector spaces, or more precisely, quaternion modules,  and quaternion bimodules.
 A systematic study of  quaternion modules is given by  Ng \cite{ng2007quaternionic}. It turns out that    the category of  (both one-sided and two-sided) quaternion Hilbert spaces is  equivalent to the category of real Hilbert spaces.  

 It is a natural question to study the   theory of octonion   spaces.  
 Goldstine and Horwitz in 1964 \cite{goldstine1964hilbert}
 initiated the study of octonion   Hilbert spaces; more recently, Ludkovsky \cite{ludkovsky2007algebras,ludkovsky2007Spectral} studied the algebras of operators in octonion  Banach spaces and spectral representations  in  octonion Hilbert  spaces. 
Although there are few results about the theory of octonion Hilbert spaces, 
  it is not full developed since it even lacks of  coherent definition of  octonion  Hilbert spaces.

  In contrast to the complex or quaternion setting,   some new phenomena occur in the setting of octonions (see Example \ref{eg:(e1,e2)=O}, \ref{eg:O+O-}):
  
  $\bullet$ If $m$ is an  element of an octonion module, then $\O m$ is not an octonion sub-module in general. 
  
  $\bullet$   If $m$ is an  element of an octonion module, then the octonion sub-module generated by $m$ maybe the whole   module.
  
  This means that the structure of  octonion module is more involved.  We point out that some gaps appear  in establishing the octonionic version of Hahn-Banach  Theorem by taking $\O m$ as a submodule   (\cite[Lemma 2.4.2]{ludkovsky2007algebras}).
  The submodule  generated by a submodule $Y$ and a point $x$ is not of the form $\{y+px\mid y\in Y,\;  p\in \O\}$,  this is wrong even for the case $Y=\{0\}$. It means the  proof can not repeat the way in canonical case.   
  The involved  structure of  octonion module  accounts for the slow developments of octonion Hilbert spaces.
  
  In the study of octonion Hilbert space and Banach space, it heavily depends on the direct sum structure of  the space under considered sometimes, which always brings the question  back to the classic situation.   For example, in the proof of  \cite[Theorem 2.4.1]{ludkovsky2007algebras}, it declares that every $\O$-vector space is of the following form:
  $$X=X_0\oplus X_1e_1\oplus\cdots\oplus X_{7}e_7.$$
  Note that therein the definition of $\O$-vector space is  actually a left $\O$-module with an irrelevant right $\O$-module structure. We thus can only consider the left $\O$-module structure of it. We show that the assertion above  does not always work.  In order to study octonion Hilbert spaces, we need to provide its solid algebraic  foundation by studying the deep structure of one-sided $\O$-modules and $\O$-bimodules.  
  We  only consider the left $\O$-modules in this paper, the bimodule case will be discussed in a later paper.
  In this paper, we characterize the structure of $\spo$-modules completely. 

We remark that Eilenberg \cite{eilenberg1948extensions} initiated the study of bimodules over non-associative rings. Jacobson studied the structures of bimodules over Jordan algebra and alternative algebra \cite{jacobson1954structure}.  One-sided modules over octonion  was  investigated  in \cite{goldstine1964hilbert} for studying octonion Hilbert spaces. However for the classification of $\spo$-modules, the problem is untouched.

  In this article, we shall give the classification of $\spo$-modules.
  It turns out that the set $\O$  admits two distinct $\O$-module structures. One is the canonical one, denoted by $\O$; the other is 
 denoted  by $\conjgt{\O}$ (see Example \ref{eg:simple O}). 
 To characterize  any $\spo$-module ${M}$, we need to introduce some new  notions, called  \textbf{associative element} and \textbf{conjugate associative element}. Their collections  are  denoted by $\huaa{M}$ and $\hua{A}{-}{M}$ respectively. The ordered pair of their dimensions as real vector spaces is called the type of $M$. These concepts are crucial in the classification  of left $\O$-modules.

  Our first main result is about the characterization of the left $\O$-modules.
  \begin{thm}\label{thm:strc thm}
  	Let $M$ be a left $\O$-module. Then  $$M=\spo\huaa{M}\oplus {\spo}\hua{A}{-}{M}.$$
  	If $\dim_{\R}{M}<\infty$, then
  	\begin{equation*}
  	M\cong \O^{ n_1}\oplus \overline{\spo}^{n_2},
  	\end{equation*}	 where  $(n_1,n_2)$ is the \type\ of $M$.
  \end{thm}

 Its proof depends heavily on the isomorphism
 between the category $O$-$\mathbf{Mod}$  and the category $\C$-$\mathbf{Mod}$. By the matrix realization of $\C$, $$\clifd{7}\cong M(8,\spr) \oplus M(8,\spr),$$   there are only two kinds of simple left $\spo$-module, namely, $\spo$ and $\overline{\spo}$ up to isomorphism. Hence the structure of finite dimension $\O$-modules follows by  Wedderburn's Theorem for central simple algebras \cite{rotman2017advancedalg}.   
  The general case relies on  some elementary properties of $\huaa{M}$ and $\hua{A}{-}{M}$, along with   an important fact that every element of $\O$-module generates a finite dimensional submodule.
  By the way, we found that the  octonions $\O$ can be endowed with  both  $\C$-module structure and $C\ell_6$-structure, using it we get an irreducible complex representation of $Spin(7)$. 

 
 Our second topic is about the cyclic elements. In contrast to the setting of complex numbers and quaternions,  cyclic elements play  key roles in the study of octonion sub-modules.
  An element $m$ in a given octonion module $M$ is called 
 \textbf{cyclic elements} if the submodule generated by it is  exactly $\O m$.
The collection of these elements is denoted by $\huac{M}$. 
It turns out that the cyclic elements are determined by the associative elements $\huaa{M}$ and the conjugate associative elements $\hua{A}{-}{M}$ completely.

 \begin{thm}\label{thm:cyclic-elements}For any left $\O$-module $M$, we have 
 $$\huac{M}=\left(\bigcup_{p\in \spo}p\cdot \huaa{M}\right)\bigcup\left(\bigcup_{p\in \spo}p\cdot \hua{A}{-}{M}\right). $$
 \end{thm}

 In view of Theorems \ref{thm:strc thm} and Theorem \ref{thm:cyclic-elements}, we find
 $$M=\text{Span}_\R \huac{M}.$$
This means that the module $M$ is determined completely by its cyclic elements in the form of real linear combination.  
For any element $m$ in a  left $\O$-module $M$,  there exist $m^\pm\in \text{Span}_\R\hua{C}{\pm}{M}$ such that $m=m^++m^-$.  We can decomposite $m^{\pm}$ into  a combination of real linearly independent cyclic elements. Denote by $l_m^{\pm}$ the minimal length of the decompositions of $m^{\pm}$. \textbf{We conjecture that,} $$\generat{m}_\O\cong \O^{l_m^+}\oplus\O^{l_m^-}.$$
If the conjecture is right, then the structure of the submodule generated by one element is completely clear.

\section{Preliminaries}

\subsection{The algebra of the octonions   $\spo$}
The  algebra of the octonions   $\spo$  is  a non-associative, non-commutative, normed division algebra over the $\spr$. Let     $e_1,\ldots,e_7$ be its natural  basis throughout this paper, i.e., $$e_ie_j+e_je_i=-2\delta_{ij},\quad i,j=1,\ldots,7.$$ For convenience, we  denote  $ e_0=1$.

In terms of the natural basis, an element in octonions can be written as $$x=x_0+\sum_{i=1}^7x_ie_i,\quad x_i\in\spr.$$
The conjugate octonion of $x$ is defined by  $\overline{x}=x_0-\sum_{i=1}^7x_ie_i$, and the norm of $x$ equals $|x|=\sqrt{x\overline{x}}\in \spr$, the real part of $x$ is $\re{x}=x_0=\frac{1}{2}(x+\overline{x})$.

 The full multiplication table is conveniently encoded in  the Fano mnemonic graph (see \cite{baez2002octonions,wang2014octonion}).
In the Fano mnemonic graph, the vertices are labeled by  $1, \ldots, 7$
instead of $e_1,\ldots, e_7$. Each of the 7 oriented lines gives a quaternionic triple. The
product of any two imaginary units is given by the third unit on the unique line
connecting them, with the sign determined by the relative orientation.

\

\noindent \small{\textbf{Fig.1} Fano mnemonic graph}
\begin{flushright}
\centerline{\includegraphics[width=4cm]{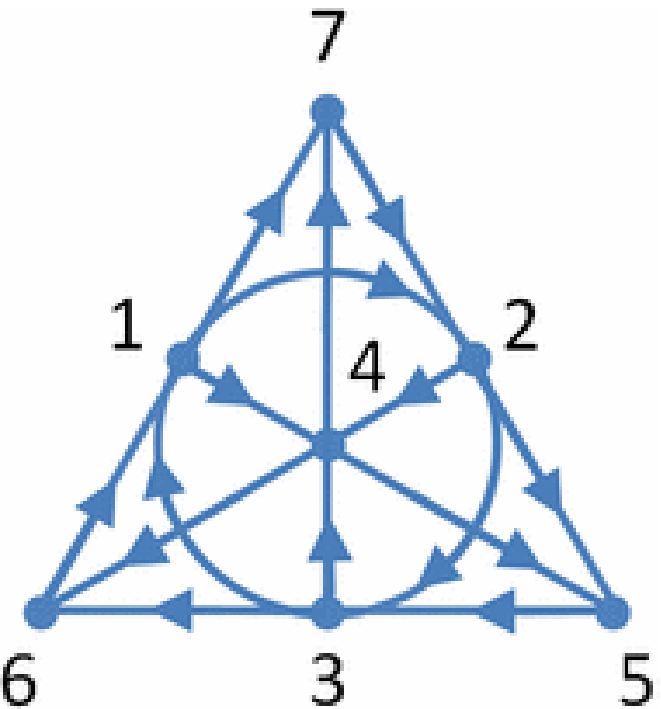}}
\end{flushright}
\normalsize

The associator of three octonions is defined as $$[x,y,z]=(xy)z-x(yz)$$
for any   $x,y,z\in\spo$, which is alternative in its arguments and has no real part. That is, $\spo$ is an alternative algebra and hence it satisfies the so-called R. Monfang identities \cite{schafer2017introduction}:
$$(xyx)z=x(y(xz)),\ z(xyx)=((zx)y)x,\ x(yz)x=(xy)(zx).$$ The commutator is defined as $$[x,y]=xy-yx.$$

\subsection{Universal Clifford algebra}
We shall use  the Clifford algebra $\C$ to study left $\O$-modules. In this subsection, we review some basic facts for universal Clifford algebras. The  Clifford algebras are introduced  by Clifford  in 1882.  For its recent development, we refer to \cite{Atiyah1964Cliffordmodules,gilbert1991clifford,Sommen2012spinor}.


\begin{mydef}
	Let $\mathbb{A}$ be an associative algebra over $\mathbb{R}$ with unit $1$ and let  $v:\R^n\rightarrow \A $ be an $\R$-linear embedding. The pair $(\A,v)$ is said to be a Clifford algebra over $\R^n$, if
	\begin{enumerate}
		\item $A$ is generated as an algebra by $\{v(x)\mid x\in \R^n\}$ and $\{\lambda1\mid \lambda\in\R\}$;
		\item $(v(x))^2=-\abs{x}^2, \forall\ x\in \R^n$.
	\end{enumerate}
\end{mydef}
We need some notations and conventions.
\begin{itemize}
	\item $x=(x_1,\ldots,x_n)\in \R^n, \abs{x}^2=\sum_{i=1}^{n}x_i^2$.
	\item\label{g_i} Let $\{f_i\}_{i=1}^n$ be the canonical orthonormal  basis of $\R^n$, $g_i=v(f_i)\in \A$, and
	$ f_i=(0,\ldots,0,1,0,\ldots,0)$   {with $1$ in the $i$-th slot.}
	\item Let $\mathcal{P}(n)$ be the collection  of all the  subsets of $\{1,\ldots,n\}$.
	\item For any $\alpha\in \P$, if $\alpha\neq \emptyset$, we write $\alpha=\{\alpha_1,\ldots,\alpha_k\}$ with $1\le \alpha_1<\cdots<\alpha_k\le n$ and we set  $g_{\alpha}=g_{\alpha_1}\cdots g_{\alpha_k}$. Otherwise, we denote $g_{\emptyset}=1$.
\end{itemize}
The Clifford algebra $\mathbb A$ can be described alternatively with the above  notations as
\begin{enumerate}
	\item $\mathbb A$ is $\R$-linearly generated  by $\{g_\alpha\mid \ \alpha\in \P\}$;
	\item $g_ig_j+g_jg_i=-2\delta_{ij}$  for any $i,j=1,\ldots,n$.
\end{enumerate}

It is well-known that  $\dim_{\R}\A\le2^n$. The Clifford algebra $(\A,v)$ over $\R^n$ may be not unique (up to isomorphism of algebras); see for example \cite{gilbert1991clifford}.
But the universal Clifford algebra $C\ell_n$ over $\R^n$   is unique up to isomorphism.

\begin{mydef}\label{univerasl}
	A Clifford algebra $(\A,v)$ is said to be a universal Clifford algebra, if for each Clifford algebra $(\mathbb{B},\mu)$ over $\R^n$, there exists an algebra homomorphism $\beta:\A\rightarrow \mathbb{B}$, such that $\mu=\beta\circ v$ and $\beta(1_\A)=1_{\mathbb{B}}$. Namely, the following diagram commutes.
	
	\begin{center}
		\usetikzlibrary{matrix,arrows}
		\begin{tikzpicture}[description/.style={fill=white,inner sep=2pt}]
		\matrix (m) [matrix of math nodes, row sep=3em,
		column sep=2.5em, text height=1.5ex, text depth=0.25ex]
		{ \mathbb{A} & & \mathbb{B} \\
			& \mathbb{R}^n & \\ };
		\path[->,font=\scriptsize]
		(m-1-1) edge node[auto] {$ \beta $} (m-1-3)
		(m-2-2) edge node[auto] {$ v $} (m-1-1)
		edge node[auto] {$ \mu $} (m-1-3);
		\end{tikzpicture}
		
	\end{center}
\end{mydef}
We recall  some equivalent descriptions of universal Clifford algebra $C\ell_n$.
\begin{thm}[\cite{gilbert1991clifford}]
	$(\A,v)$  is a Clifford algebra over $\R^n$, the following are equivalent:
	\begin{enumerate}
		\item $(\A,v)$ is a universal Clifford algebra $C\ell_n$ over $\R^n$;
		\item $dim_{\R}\A=2^n$;
		\item $g_1\cdots g_n\notin \R$.
	\end{enumerate}
\end{thm}
At the last of this subsection, we give   the following algebra isomorphism of $C\ell_n$ (see for example \cite{Atiyah1964Cliffordmodules}).
\begin{itemize}
	\item $C\ell_{n+8}\cong C\ell_n\otimes M(16,\R)\cong M(16,C\ell_n)$;
	\item for $n=0,\ldots,7$ we have table
	\begin{center}
		\begin{tabular}[ht]{c c}
			\hline
			n & $C\ell_n$ \\
			\hline
			0 & $\R$ \\
			1 & $\mathbb{C}$ \\
			2 & $\mathbb{H}$ \\
			3 & $\mathbb{H}\bigoplus\mathbb{H}$ \\
			4 & $M(2,\mathbb{H})$ \\
			5 & $M(4,\mathbb{C})$ \\
			6 & $M(8,\R)$ \\
			7 & $M(8,\R)\bigoplus M(8,\R)$ \\
			\hline
		\end{tabular}
	\end{center}
\end{itemize}
Here, we denote by $M(k, \mathbb F)$ the collection of all $k\times k$ matrices  with each entry in the algebra $\mathbb F$.

\section{$\spo$-modules}
We set up in this section some preliminary definitions and results on left $\O$-modules. 


\begin{mydef}
	 A real vector space $M$ is called a (left) \textbf{$\O$-module},  equipped with a scalar multiplication $\O\times M\rightarrow M$, denoted by $$(q,m)\mapsto qm,$$
	 such that the following axioms hold for all $q,q_1,q_2\in \O,\; \lambda\in \R$ and all $ m,m_1,m_2\in M$:
	\begin{enumerate}
		\item $(\lambda q)m=\lambda(qm)=q(\lambda m)$;
		\item $(q_1+q_2)m=q_1m+q_2m, \ q(m_1+m_2)=qm_1+qm_2$;
		\item $[q_1,q_2,m]=-[q_2,q_1,m];$
		\item $1m=m.$
	\end{enumerate}
	Here, the \textbf{left associator} is defined by $$[q_1,q_2,m]:=(q_1q_2)m-q_1(q_2m).$$
\end{mydef}

Note that this definition  is  equivalent to the definition given in \cite{ludkovsky2007algebras,ludkovsky2007Spectral}, wherein the axiom  $(\mathrm{\romannumeral3)}$ is replaced by  $$q^2m=q(qm), \text{ for all } q\in \O, m\in M.$$
The proof is trivial by polarizing the above relation.
It also agrees with the one given in \cite{goldstine1964hilbert} wherein $M$  needs to satisfy an additional axiom: $p(p^{-1}x)=x$, which can be deduced from the equation $p(px)=p^2x$ directly.

Let $M$ be a left $\O$-module.
The definition of the terms submodule, homomorphism, isomorphism, kernel of a homomorphism,  which are familiar from the study of associative modules, do not invole associativity of multiplication and are thus immediately applicable to the case in general.  Let  $\Hom_{\O}(M,M')$ denote the set of all $\O$-homomorphisms from $M$ to $M'$ as usual ($M'$ being arbitrary $\O$-module).
So is the notation $JN$ for the subset of $M$
spaned by all products $rn$ with $r\in J$ and $n\in N$
($J$ being arbitrary nonempty subset of $\O$ and $N$ being arbitrary nonempty subset of $M$), here we must of course distinguish between $J_1(J_2N)$ and $(J_1J_2)N$. Let $\left\langle N\right\rangle _\O$ denote the minimal submodule which contains $N$ as before. 
An element $m\in M$ is said to be \textbf{associative} if $$[p,q,m]=0, \quad \forall p,q\in \O.$$ Denote by $\huaa{M}$ the set of all associative elements in $M$:
$$\huaa{M}:=\{m\in M\mid [p,q,m]=0,\ \forall p,q \in \O\}.$$
One useful identity  which holds in any $\O$-module $M$ is
\begin{equation}\label{eq:[p,q,r]m+p[q,r,m]=[pq,r,m]-[p,qr,m]+[p,q,rm]}
[p,q,r]m+p[q,r,m]=[pq,r,m]-[p,qr,m]+[p,q,rm],
\end{equation}
where $p,q,r\in \O,\ m\in M$. The proof is  by striaghtforward calculations. It follows that:
\begin{lemma}\label{lem: [p,q,rm]=[p,q,r]m}
	For all  associative element $m\in \huaa{M}$, we have $$[p,q,rm]=[p,q,r]m,  \text{ for all } p,q,r \in \O .$$
%
	
\end{lemma}

The following elementary property  will be useful in the sequel. The proof is trivial and will be omitted here.
\begin{prop}\label{prop:f(huaaM) in huaaN}
	If $f\in \Hom_{R}(M,N)$, then $f([p,q,x])=[p,q,f(x)]$ for all $p,q\in \O,\ x\in M$. Therefore  $f(\huaa{M})\subseteq \huaa{N}$.
\end{prop}

\

We give several elemetary left $\O$-module examples.
\begin{eg}
	It is easy to see the real vector spaces  $\spo,\; \spo^n,\; M(n,\spo)$ with the obvious multiplication are all left  $\O$-module. Clearly,  the sets of associative elements on these modules    are  $\spr,\;\spr^n,\;M(n,\spr)$ respectively.
\end{eg}

We can define  a different $\spo$-module structure on the octonions $\spo$ itself.
\begin{eg}[	$\overline{\spo}$]\label{eg:simple O}
	 Define: $$p\hat{\cdot}x:=\overline{p}x,\ \forall p\in \spo,x\in \spo.$$
	It's easy to check this is  a left $\spo$-module. Indeed $$p^2\hat{\cdot}x=\overline{p}^2x=p\hat{\cdot}(p\hat{\cdot}x).$$
	 We shall denote this $\spo$-module by 	$\overline{\spo}$. By direct calculations, we obtain: $$[p,q,x]_{\overline{\O}}=[p,q,x]+\overline{[p,q]}x.$$ This implies that $\huaa{\overline{\spo}}=\{0\}$. Note that Proposition \ref{prop:f(huaaM) in huaaN} ensures that $\huaa{M}\cong_\spr \huaa{N}$ when  $M\cong_\O N$, and therefore $\spo\ncong\overline{\spo}$.
\end{eg}
However, there is a special subset in $\overline{\spo}$, that is, the real subspace $\spr$. To describe such elements, we introduce a new notion of {conjugate associative element}.
\begin{mydef}
	An element $m\in M$ is said to be \textbf{conjugate associative} if
	$$(pq)m=q(pm),\quad \forall p,q\in \spo.$$
	Denote by $\hua{A}{-}{M}$ the set of all conjugate associative elements.
\end{mydef}

\begin{lemma}\label{lem:hua-(conjgt O)=R}
$\hua{A}{-}{\overline{\spo}}=\spr.$
\end{lemma}
\begin{proof}
	Suppose $x\in \hua{A}{-}{\overline{\spo}}$, then for any $p,q\in \spo$,
	$$0=(pq){\hat{\cdot}}x-q{\hat{\cdot}}(p{\hat{\cdot}}x)=(pq){\hat{\cdot}}x-(qp){\hat{\cdot}}x+[q,p,x]_{\overline{\O}}=[p,q]{\hat{\cdot}}x-[p,q,x]_{\overline{\O}}$$
	Hence we obtain
	$$\overline{[p,q]}x-[p,q,x]-\overline{[p,q]}x=[q,p,x]=0$$
	this implies that $x\in \spr$. Clearly $\spr \subseteq \hua{A}{-}{\overline{\spo}}$, thus $\hua{A}{-}{\overline{\spo}}=\spr.$
\end{proof}

\begin{lemma}\label{lem:huaa cap huaa-=0}
	For any left $\O$-module $M$, we have $\huaa{M}\cap \hua{A}{-}{M}=\{0\}$ .
\end{lemma}
\begin{proof}
	Obviously $0\in \huaa{M}\cap \hua{A}{-}{M}$. Let $ x\in \huaa{M}\cap \hua{A}{-}{M}$, then for any $p,q\in \spo$,
	$$[p,q]x=(pq)x-(qp)x=(pq)x-p(qx)=[p,q,x]=0.$$
	This implies $x=0$ since we can choose $p,q\in\O$ such that $[p,q]\neq 0$. This proves the lemma.
\end{proof}

\begin{rem}
Clearly, both $\huaa{M}$ and $\hua{A}{-}{M}$ are real vector spaces. If $M$ is of finite dimension, we call the ordered pair $(\dim_{\R}\huaa{M},\dim_{\R}\hua{A}{-}{M})$ \textbf{\type} of $M$. 	 
	We shall use these notions to describe the structure of left $\O$-modules. It turns out that type is a complete invariant in the finite dimensional case.
\end{rem}
%
%

Let $M$ be a left $\spo$-module. We shall establish some  properties of  associative elements and conjugate associative elements which will be used in the seuel.

%


%
%

\begin{lemma}\label{lem:free mod unique represt}
	Let $ \{x_i\}_{i=1}^n$ be an  $\spr$-linearly independent set of \asselm s of  $M$. If $$\sum_{i=1}^nr_ix_i=0 ,\ r_i\in \spo \text{ for each }i=1,\dots, n,$$    then $r_i=0 $ for each $i=1,\dots, n.$
\end{lemma}
\begin{proof}
	The proof is by induction on $n$.  For the case $n=1$, we have  $rx=0$. If $r\neq 0$,  since $x\in \huaa{M},$ it follows that $$0=r^{-1}(rx)=(r^{-1}r)x=x,$$ a contrdiction with our assumption $x\neq 0$. Assume the lemma holds for degree $k$, we will prove it for $k+1$. Suppose $\sum_{i=1}^{k+1}r_ix_i=0$ and $r_{k+1}\neq 0$.
	Denote $s_i=-r_ir_{k+1}^{-1}$. Therefore
	$$x_{k+1}=\sum _{i=1}^k s_ix_i.$$
	Since $x_{k+1}\in \huaa{M}$, thus  for all $p,q,\in \O$,
	$$0=[p,q,x_{k+1}]=\sum_{i=1}^k[p,q,s_ix_i]=\sum _{i=1}^k[p,q,s_i]x_i$$
	where we  heve used Lemma \ref{lem: [p,q,rm]=[p,q,r]m} in  the last eqution. Hence by  induction hypothesis we conclude that,
	$$[p,q,s_i]=0, \text{ for each } i\in\{1,\dots, k\}\;  \text{ and for all } p,q\in \spo.$$ This implies $s_i\in \spr$. Note that $$\sum_{i=1}^ks_ix_i=x_{k+1},$$
	which  contradicts the hypothesis that $ \{x_i\}_{i=1}^n$ is  $\spr$-linearly independent, we thus	prove the lemma.
\end{proof}
Note that we have actually proved the following property.
\begin{cor}\label{{cor:O linear indpt= R linear indpt}}
	Let $S\subseteq \huaa{M}$. Then that $S$ is $\spo$-linearly independent if and only if it is $\spr$-linearly independent.
	
\end{cor}
\begin{lemma}\label{lem: xishu asselm}
	Under the assumptions of Lemma \ref{lem:free mod unique represt}, if $y=\sum_{i=1}^nr_ix_i\in \huaa{M}$, then we have  $r_i\in \spr  \text{ for each }i\in \{1,\dots, n\}.$
\end{lemma}
\begin{proof}
	 Since $y\in \huaa{M}$, we have for all $p,q\in \spo$,
	$$0=[p,q,\sum_{i=1}^nr_ix_i]=\sum_{i=1}^n[p,q,r_i]x_i.$$
	By Lemma \ref{lem:free mod unique represt}, we conclude $[p,q,r_i]=0,\ i=1,\dots, n$. This yields $r_i\in \spr$.
\end{proof}

We next consider the properties of conjugete elements. It turns out that similar statements hold for $S\subseteq \hua{A}{-}{M}$.

\begin{lemma}\label{lem:hua- xishu }
	Let $ S=\{x_i\}_{i=1}^n\subseteq \hua{A}{-}{M}$ be an  $\spr$-linearly independent set. If $$\sum_{i=1}^nr_ix_i=0 ,\ r_i\in \spo \text{ for each }i=1,\dots, n,$$    then $r_i=0 $ for each $ i=1,\dots, n.$ 
	
	Moreover, if $y=\sum_{i=1}^nr_ix_i\in \hua{A}{-}{M}$, then $r_i\in \spr$ for each $ i=1,\dots, n.$
\end{lemma}
\begin{proof}
	Induction on $n$ as before. The case of $n=1$ is trivial. Assume the lemma holds for degree $k$, we will prove it for $k+1$. Suppose $\sum_{i=1}^{k+1}r_ix_i=0$ and $r_{k+1}\neq 0$.
	Denote $s_i=-r_ir_{k+1}^{-1}$. Therefore
	$$x_{k+1}=\sum _{i=1}^k s_ix_i.$$
	Since $x_{k+1}\in \hua{A}{-}{M}$, thus for all   $ p,q,\in \O$,
	\begin{align*}
	0&=(qp)x_{k+1}-p(qx_{k+1})\\
	&=\sum_{i=1}^k(qp)(s_ix_i)-p(q(s_ix_i))\\
	&=\sum_{i=1}^k (s_i(qp))x_i-p((s_iq)x_i)\\
	&=-\sum _{i=1}^k\left([s_i,q,p]\right)x_i
	\end{align*}
	Hence by  induction hypothesis we obtain that $s_i\in \spr$.
	The rest of the proof runs much the same as in Lemma \ref{lem:free mod unique represt}.
\end{proof}

\section{The structure of left $\spo$-moudles}
In this section, we are in a position to  formulate the structure of any left $\O$-module.  We will first concerned with the finite dimensional case and then the general case follows.
\subsection{Finite dimensional $\spo$-modules}
It is well-known (for example, \cite{baez2002octonions,harvey1990spinors}) that the octonions have a very close relationship with spinors in $7,8$ dimensions. In particular, multiplication
by imaginary octonions is equivalent to Clifford multiplication on spinors in $7$ dimensions.
It turns out that 	the category of left $\spo$-modules is isomorphic to the category of left $\clifd{7}$-modules.

For any left $\spo$-module $M$, according to $[e_i,e_j,x]=-[e_j,e_i,x]$, we get the left multiplication operator $L$ satisfies:
$$L_{e_i}L_{e_j}+L_{e_j}L_{e_i}=-2\delta_{ij}Id.$$
Hence $\mathbb{A}:=Span_\spr \{L_{e_i}\mid i=0,1,\ldots,7  \}$ is a Clifford algebra over $\spr^7$.
This yields a $\clifd{7}$-module structure on $M$, because  from  the universal properties of $Cl_7$, we have a non-trivial ring homomorphism $$\rho:\clifd{7}\rightarrow \mathbb{A}\rightarrow \text{End}_\spr(M).$$ Denote this  $\clifd{7}$-module by ${}_{\clifd{7}}M$, or just $M$, and the $\clifd{7}$-scalar multiplication is given  by
$$ g_{\alpha}m: = e_{\alpha_1}(e_{\alpha_2}(\cdots(e_{\alpha_k}m)))$$ for any  $\alpha\in \P$.
Here $g_{\alpha}=g_{\alpha_1}\cdots g_{\alpha_k}:=L_{e_{\alpha_1}}\cdots L_{e_{\alpha_k}}$.
Let $f\in \Hom _\spo(M,M')$ be a left $\spo$-homomorphism, where $M,M'$ are two left $\spo$-modules. Then $$f(L_{e_i}x)=f(e_ix)=e_if(x)=L_{e_i}f(x).$$
and hence $f(g_\alpha x)=g_\alpha f(x)$. This means $f\in \Hom_{\clifd{7}}(M,M')$.
Conversely, for any left $\clifd{7}$-module $M$, let $\{g_i\}_{i=1}^7$ be a basis in $\spr^7$. Define:
$$e_i\hat{\cdot}x:=g_ix.$$
Then $$e_i\hat{\cdot}(e_j\hat{\cdot}x)=g_i(g_jx)=(g_ig_j)x=(-2\delta_{ij}-g_jg_i)x=-2\delta_{ij}x-e_j\hat{\cdot}(e_i\hat{\cdot}x)$$
This implies that $$e_i\hat{\cdot}(e_j\hat{\cdot}x)+e_j\hat{\cdot}(e_i\hat{\cdot}x)=(e_ie_j+e_je_i)\hat{\cdot}x$$
by transposition of terms, we obtain$$[e_i,e_j,x]=-[e_j,e_i,x].$$
This yields for any $p,q\in \spo$, $[p,q,x]=-[q,p,x]$. Consequently
$M$ is a left $\spo$-module. For any $f\in \Hom_{\clifd{7}}(M,M')$, $$f(e_i\hat{\cdot}x)=f(g_ix)=g_if(x)=e_i\hat{\cdot}f(x)$$
Therefore $f\in \Hom_\spo(M,M')$.
In summary, we  get the following important result:
\begin{thm}\label{thm:O mod= Cl7 mod}
	The category of left $\spo$-module is isomorphic to the category of left $\clifd{7}$-module. Moreover, the only two kinds of simple $\spo$-module are  $\spo$ and $\overline{\spo}$ up to isomorphsim.
\end{thm}
\begin{proof}
	 Naturaly we have two categories $\O$-$\mathbf{Mod}$ and $\C$-$\mathbf{Mod}$.
	 \begin{eqnarray*}
	 	T\colon\O\text{-}\mathbf{Mod} &\longrightarrow& \C\text-\mathbf{Mod}
	 	\\
	 	{}_{\O}M &\mapsto& {}_{{\C}}M\end{eqnarray*}
	 and for any morphism $\varphi \in \text{Hom}_\O(M,N)$, it maps to  $T(\varphi)\colon{}_{{\C}}M\to {}_{{\C}}N$, which is given by $$T(\varphi):m\mapsto \varphi(m).$$
	
 Clearly, this is an isomorphism by above discussion.  As is well known, $\clifd{7}$ is a semi-simple algebra and $$\clifd{7}\cong M(8,\spr) \oplus M(8,\spr),$$ 
	the only simple $\clifd{7}$-module is $(\spr^8,0)$ and $(0,\spr^8)$ up to isomorphism \cite{gilbert1991clifford}. Hence there  also only exist two kinds of simple $\spo$-module. In view of   Example \ref{eg:simple O}, we thus conclude that  $\spo$ and $\overline{\spo}$ are the two different   kinds of simple $\spo$-modules. This completes the proof.
\end{proof}
\begin{cor}\label{cor:fd O-M}
If $M$ is of finite real dimension, then
 \begin{equation}\label{iso}
M\cong \O^{ n_1}\oplus \overline{\spo}^{n_2}.
\end{equation}	 where  $(n_1,n_2)$ is the \type\ of $M$. In particular, $\mathrm{dim}_\spr M=8(n_1+n_2)$.

\end{cor}
\begin{proof}
	By Theorem \ref{thm:O mod= Cl7 mod}, we can regard $M$ as a $\clifd{7}$-module ${}_{\C}M$.
	It is known that $\clifd{7}$ is a semi-simple algebra, we have
	$$\clifd{7}\cong M(8, \R)M \oplus M(8, \R).$$	 
	This means that $\clifd{7}$ has exactly two isomorphism classes of simple $\clifd{7}$-
	modules (see for example \cite[Chap5, Thm10]{Alperin2012gtm162}).
	In view of Wedderburn's Theorem for central simple algebras therefore, we conclude $${}_{\C}M\cong {}_{\C}S_1^{n_1}\oplus {}_{\C}S_2^{n_2}$$
 where $S_1,S_2$ are  the representation  of the two isomorphism classes of simple $\C$-modules.
 Thus by Theorem \ref{thm:O mod= Cl7 mod}, $${}_{\O}M\cong{}_{\O}S_1^{ n_1}\oplus{}_{\O}S_2^{ n_2}.$$  In particular, let $S_1=\spo,S_2=\overline{ \spo}$, then we get the conclusion as desired. We next show that $(n_1,n_2)$ is just the type. By definition, one can prove $\huaa{M\oplus N}=\huaa{M}\oplus\huaa{N}$ and $\hua{A}{-}{M\oplus N}=\hua{A}{-}{M}\oplus\hua{A}{-}{N}$. Then it follows from  Lemma \ref{lem:hua-(conjgt O)=R} that
 $\dim_{\R}\huaa{\O^{ n_1}\oplus \overline{\spo}^{n_2}}=n_1$ and $\dim_{\R}\hua{A}{-}{\O^{ n_1}\oplus \overline{\spo}^{n_2}}=n_2$. This completes the proof.
\end{proof}

We can  give a complete description of the set of  homomorphisms between two finite dimensional left $\O$-modules.
\begin{thm} Let $M, N$ be  two left $\O$-modules of finite dimension. Suppose the type of $M$ and $N$ are $(m_1,m_2),\; (n_1,n_2)$ respectively. Then
	\begin{equation}\label{mod-matrix-rep}
	 \text{Hom}_{\O}(M,N)\cong M_{n_1\times m_1}(\R)\bigoplus M_{n_2\times m_2}(\R).
	 \end{equation}
\end{thm}

\begin{proof}

It follows from Corollary \ref{cor:fd O-M} that,
\begin{equation}\label{mod-matrix-rep55} M\cong \O^{m_1}\oplus\conjgt{\spo}^{ m_2}, \qquad N\cong  \O^{n_1}\oplus \conjgt{\spo}^{n_2}.
\end{equation}
We claim each homomorphism of ${Hom}_{\O}(M,N)$ is of the form
$$\left(
	\begin{array}{cccccc}
	\theta_{11} & \ldots & \theta_{1m_1} & 0 & 0 & 0 \\
	\vdots & \ldots & \vdots & 0 & 0 & 0 \\
	\theta_{n_11} & \ldots & \theta_{n_1m_1} & 0 & 0 & 0 \\
	0 & 0 & 0 & \psi_{11} & \ldots & \psi_{1m_2} \\
	0 & 0 & 0 & \vdots & \ldots & \vdots \\
	0 & 0 & 0 & \psi_{n_21} & \ldots & \psi_{n_2m_2}\\
	\end{array}
	\right).$$
	Inneed,  we have  $\text{Hom}_{\O}(\O,\O)=\R  \; \text{Id}_{\O}$,  $\Hom_\spo(\overline{ \spo},\overline{ \spo})= \R \;\text{Id}_{\O}$, $\Hom_\spo(\spo,\overline{ \spo})\cong \{0\}$.
	Notice that for any $\varphi\in \text{Hom}_{\O}(\O,\O)$ we have $$\varphi(q)=q\varphi(1),$$  we only need  to show that $\varphi(1)\in \mathbb{R}$.    Propositon \ref{prop:f(huaaM) in huaaN} then yields  the conclusion as required.
%
	As for the case   $ \Hom_\spo(\overline{\spo},\overline{\spo})\cong\spr $. Given $f\in \Hom_\spo(\overline{\spo},\overline{\spo})$, write $f(1)=r$, then $$f(x)=f(\overline{x}\hat{\cdot}1)=\overline{x}\hat{\cdot}f(1)=xr.$$ Since $f$ is an $\spo$-homomorphism, it follows that for all $p,x\in \spo$,
	$$(p\hat{\cdot}x)r=f(p\hat{\cdot}x)=p\hat{\cdot}f(x)=p\hat{\cdot}(xr).$$
	This yields $[\overline{p},x,r]=0$ for all $p,x\in \spo$, consequently $r\in \spr$ as desired.
	Finally, it follows from Schur's Lemma that
	$$\Hom_\spo(\spo,\overline{ \spo})\cong \Hom_{\clifd{7}}({}_{\clifd{7}}\spo,{}_{\clifd{7}}\overline{ \spo})=\{0\}.$$
This completes the proof.
\end{proof}

\


At last, we point out another relation  between  $Spin(7)$ and  octonions.  More precisely, we can provide a realization of spinor space over $Spin(7)$ in terms of octonions $\O$.

Here  we refer to the concept of (Weyl) spinor spaces as  the irreducible complex representations of $Spin(p,q)$; see ~\cite[Chap. $\uppercase\expandafter{\romannumeral 1}$ sec. 4]{Sommen2012spinor}. By definition,
\[Spin(7)=\{v(x_1)v(x_2)\cdots v(x_{2k}):x_i\in \R^7,\abs{x_i}=1,i=1,\ldots,2k\}\]
It is well-known that 
$$ Spin(7) \subset\C^{+},$$
where $\C^+$ is $\R$-linearly generated by $\{g_{\alpha}\mid\alpha\in \P,\abs{\alpha}=2k\}$.

Let $\mathbb{C}l_7=\C\otimes \mathbb{C}$ be the complexification of $\C$. It is a complex Clifford algebra and there exists a $\mathbb{C}$-algebra isomorphism
\[\mathbb{C}l_7\cong M(8,\mathbb{C})\oplus M(8,\mathbb{C}).\]

Notice that  ${\mathbb{C}l_7}^+$ is $\mathbb C$-linearly generated by $Spin(7)$ and $${\mathbb{C}l_7}^+\cong \mathbb{C}l_6\cong M(8,\mathbb{C})$$ as $\mathbb C$-algebra. This means that ${\mathbb{C}l_7}^+$  is a simple algebra. Therefore, 
$Spin(7)$  has only one  irreducible representation. We denote by  $S_6$ the irreducible representation of  $Spin(7)$. It is also a $\mathbb{C}l_6$-module.

As  seen before, $\O$ has a $C\ell_7$-module structure.  Thanks to 
\begin{eqnarray}\label{Cl_6}
L_{e_1}\cdots L_{e_7}=-\text{Id}
\end{eqnarray}
it admits  a $C\ell_6$-module structure ${}_{C\ell_6}{\O}$, too.
Indeed,   let $\A$ be the algebra generated by $\{L_{e_i}\mid i=1,\ldots,7\}$.  In virtue of  \eqref{Cl_6}, we know that $\A$ is not the universal Clifford algebra over $\R^7$, so that
it is isomorphic to $C\ell_6$; see \cite{gilbert1991clifford}. Consequently,  we have a ring homomorphism
\[C\ell_6\cong \A \hookrightarrow \text{End}_{\R}(\O), \]
which provides a $C\ell_6$-module structure ${}_{C\ell_6}\O$ on $\O$.

As a result,   ${}_{C\ell_6}\O\otimes \mathbb{C}$ is a simple $\mathbb{C}l_6$-module, which gives the realization of spinor space over $Spin(7)$.

\subsection{Structure of general left $\spo$-modules}
In this subsection, we proceed to study the  characterization of the general  left $\O$-modules.
We first introduce the notion of basis on $\O$-modules.
\begin{mydef}
	Let $M$ be a left $\spo$-module. A subset  $S\subseteq M$ is called a \textbf{basis} if $S$ is $\O$-linearly independent and $M=\O S$.
\end{mydef}
\begin{thm}\label{thm:strc of left O-mod}
	Each left $\spo$-module $M$ has a basis $S$ included by $\huaa{M}\cup \hua{A}{-}{M}$. In particular,
		 $$M=\spo\huaa{M}\oplus {\spo}\hua{A}{-}{M}.$$
	
	Moreover, let $\Lambda_1,\Lambda _2$ be two index sets satisfying $\left|\Lambda_1\right|=\left|S\cap \huaa{M}\right|$, $\left|\Lambda_2\right|=\left|S\cap\hua{A}{-}{M}\right|$, here $\left|S\right|$ stands for the cardinality of $S$, then  $M\cong(\oplus_{i\in \Lambda_1}\spo) \bigoplus (\oplus_{i\in \Lambda_2}\overline{\spo})$.
\end{thm}
Its proof will depend on the following lemma. 
\begin{lemma}\label{lem:<m> is finite dim}
	Let $M$ be a left $\spo$-module, then $\generat{m}_\spo$ is  finite dimensional for any $m\in M$. More precisey, the dimension is at most $128$. 
\end{lemma}
\begin{proof}
	$\generat{m}_\spo$ is such module generated by $e_{i_1}(e_{i_2}(\cdots (e_{i_n}m)))$, where $i_k\in\{1,2,\ldots,7\}, n\in \mathbb{N}$. Note that
	$$e_i(e_jm)+e_j(e_im)=(e_ie_j+e_je_i)m=-2\delta_{ij}m,$$
	hence the element defined by $e_{i_1}(e_{i_2}(\cdots (e_{i_n}m)))$ for $n>7$ can be reduced.
	Thus the vectors $\{m,\; e_1m,\;\ldots\;,\;e_7m,e_1(e_2)m,\;\ldots\;,e_1(e_2(\cdots (e_7m)))\}$  will generate $\generat{m}_\spo$,  we conclude that $\dim_\spr \generat{m}_\spo\leqslant C_7^0+C_7^1+\cdots+C_7^7=128$.
\end{proof}
\begin{rem}
	In fact, this property has already appeared in \cite{goldstine1964hilbert}. However, it is worth stressing the essentiality of this property. It  enables us to characterize  the structure of general left $\O$-modules in terms of finite case.
\end{rem}
\begin{proof}[Proof of Theorem \ref{thm:strc of left O-mod}]
	For the case  $\dim_\spr M< \infty$,  by Corallory \ref{cor:fd O-M}, $M\cong \spo^n \oplus \overline{ \spo}^{n'}$ for some nonnegative integers $(n,n')$. It's easy  to see that there is a basis $\{\epsilon_i\}_{i=1}^{n+n'}\subseteq \huaa{M}\cup \hua{A}{-}{M}$.  As for general case,
	let $S^+$ be a basis of the real vector space $\huaa{M}$ and $S^-$ a  basis of $\hua{A}{-}{M}$, let $S:=S^+\cup S^-\subseteq\huac{M}$. It follows from  Lemma \ref{lem:huaa cap huaa-=0} that $\huaa{M}\cap \hua{A}{-}{M}=\{0\}$,  $S$ is clearly $\spr$-linearly independent. In view of Lemma \ref{lem:free mod unique represt} and Lemma \ref{lem:hua- xishu }, we conclude that $S$ is also $\spo$-linearly independent. We next show that  $\spo{S}=M$. If not, there exists a nonzero element $m\in M\setminus\spo S$. It follows by Lemma \ref{lem:<m> is finite dim} that $\generat{m}_\spo$ has a basis $\{x_i(m)\}_{i=1}^n\subseteq\huaa{\generat{m}_\spo}\cup \hua{A}{-}{\generat{m}_\spo}\subseteq\huaa{M}\cup \hua{A}{-}{M}$, hence we can assume
	$$m=\sum_{i=1}^n r_ix_i(m),\quad r_i\in \spo.$$
	Suppose $x_i(m)=\sum_{j}r_{ij}s_j$, $s_j\in S$, and  it follows from Lemma \ref{lem: xishu asselm} and Lemma \ref{lem:hua- xishu } that $r_{ij}\in \spr$, then we obtain
	$$m=\sum (r_ir_{ij})s_j,\quad r_ir_{ij}\in \spo,\; s_j\in S,$$
	which contradicts our assumption.
	
	Let $N$ denote the $\O$-module $(\oplus_{i\in \Lambda_1}\spo) \bigoplus (\oplus_{i\in \Lambda_2}\overline{\spo})$. Canonically we can choose a basis $\{\epsilon_i\}_{i\in \Lambda_1\cup \Lambda_2}$ and satisfies $\epsilon_i\in \huaa{N}$ when $i\in \Lambda_1$, $\epsilon_j\in \hua{A}{-}{N}$ when $j\in \Lambda_2$. Let $S=\{s_i\}_{i\in \Lambda_1\cup \Lambda_2}$ be a basis of $M$ satisfying the hypothesis in theorem. We  define:
	$$f\colon M\to N, \quad \sum r_is_i\mapsto \sum r_i\epsilon_i.$$
	Then for any $p\in \O$,
	\begin{align*}
	f\left(p\sum_{\Lambda_1\cup \Lambda_2} r_is_i\right)&=f\left(\sum_{\Lambda_1} p(r_is_i)+\sum_{\Lambda_2} p(r_is_i)\right)\\
	&=f\left(\sum_{\Lambda_1} (pr_i)s_i+\sum_{\Lambda_2} (r_ip)s_i\right)\\
	&=\sum_{\Lambda_1} (pr_i)\epsilon_i+\sum_{\Lambda_2} (r_ip)\epsilon_i\\
	&=\sum_{\Lambda_1} p(r_i\epsilon_i)+\sum_{\Lambda_2} p(r_i\epsilon_i)\\
	&=pf\left(\sum_{\Lambda_1\cup \Lambda_2} r_is_i\right)
	\end{align*}
	 This shows $f\in \Hom_\O(M,N)$, similarly we can also define $g\in \Hom_\spo(N,M)$ such that $$fg=\text{Id}_N,\; gf=\text{Id}_M.$$ Hence we get the conclusion as desired.
\end{proof}
\begin{rem}
	In the study of octonion Hilbert space and Banach space, it heavily depends on the  direct sum structure  of the space under considered sometimes, which always brings the question  back to the classic situation.   For example, in the proof of Hanh-Banach Theorem in \cite[Theorem 2.4.1]{ludkovsky2007algebras}, it declares that every $\O$-vector space is of the following form:
	$$X=X_0\oplus X_1e_1\oplus\cdots\oplus X_{7}e_7.$$
  Note that therein the definition of $\O$-vector space is  actually a left $\O$-module with an irrelevant right $\O$-module structure. We thus can only consider the left $\O$-module structure of it. In view of Theorem \ref{thm:strc of left O-mod}, of course the assertion does not always work.

\end{rem}

\subsection{Cyclic elements in left $\spo$-module}


Let $M$ be a left $\spo$-module throughout this subsetion.
The submodule $\left\langle m\right\rangle _{\O}$ generated  by one point  will be very different  from the classical case.  As is known that $\O m$ is not always a submodule (see Example \ref{eg:(e1,e2)=O}, Example \ref{eg:O+O-}).  We introduce the notion of cyclic element, which  generates a simple submodule,  to describe this phenomenon. It turns out that every element is a real linear combination of cyclic elements, although the quantity of cyclic elements is much less than others.


\begin{mydef}\label{def:cir,ass}
	 An element 	$m\in M$ is said to be \textbf{cyclic} if $\left\langle m\right\rangle _\O=\O m$.
%
	Denote by  $\huac{M}$ the set of all cyclic elements in $M$.
	
\end{mydef}

%
%
\begin{prop}\label{prop:cyclic=semiass}
	An element $m$ is cyclic if and only if  for all $ r,p \in \O$, there exists  $q\in \O$, such that $  [r,p,m]=qm $.
\end{prop}
\begin{proof}
	Assume $m$ satisfies  the hypothesis above, then it's easy to check that $\O m$ is  a submodule, that is,  $m$ is cyclic.
	Assume $m$ is cyclic, thus for all $ r,p \in \O$, $r(pm)\in \O m$ and hence there exists $ s\in \O$, such that  $r(pm)=sm $, then $[r,p,m]=(rp-s)m$. This proves the lemma.
\end{proof}

Our first observation is the following lemma which is useful sometimes.
\begin{lemma}\label{lem:phuaa in huac}
		Let $M$ be a left $\spo$-module, then
$\bigcup_{p\in \spo}p\cdot \huaa{M}\subseteq\huac{M}$.
		
	\end{lemma}
	
	\begin{proof}
	 Let $0\neq x\in \huaa{M}$, we want to prove that $px\in \huac{M}$.  To see this, take  $r,s\in \spo$ arbitrarily, then using Lemma \ref{lem: [p,q,rm]=[p,q,r]m}, we obtain:
	\begin{align*}
	[r,s,px]&=[r,s,p]x\\&=[r,s,p]((p^{-1}p)x)\\&=([r,s,p]p^{-1})(px)-[[r,s,p],p^{-1},px]\\&=([r,s,p]p^{-1})(px)-[[r,s,p],p^{-1},p]x\\&=([r,s,p]p^{-1})(px)
	\end{align*}
	In view of Proposition \ref{prop:cyclic=semiass}, we thus get $px\in \huac{M}$ as desired.
	\end{proof}
\begin{rem}
	Note that the set $\bigcup_{p\in \spo}p\cdot \huaa{M}$ is not $\O \huaa{M}$. Actually,  $$\O \huaa{M}=\left\lbrace \sum_{i=1}^n p_ix_i\mid p_i\in \O,x_i\in \huaa{M},n\in \mathbb{N}   \right\rbrace .$$ 
	In fact, it is easy to know the sum of two cyclic elements may be not a cyclic element anymore.
\end{rem}
Now let us first consider the cyclic elements in case of  finite dimensional $\O$-modules. It had been seen that $(e_1,e_2,0)$ will generat a $16$ dimensional real space and $(e_1,e_2,e_3)$ will generat a $24$ dimensional real space, which means that  both are   not cyclic elements (see \cite{goldstine1964hilbert}).  In fact, this is a general  phenomenon. In  case $M=\O^2$, we actually have:

\begin{eg}\label{eg:(e1,e2)=O}
	Suppose $x=(x_1,x_2)\in M$, then $\left\langle x\right\rangle _\spo=M$ if and only if  $x_1,x_2$ are  real linearly indenpendent.

	In particular,  in view of the structure of finite dimensional $\O$-modules (see Corollary \ref{cor:fd O-M}), we obtain:
		\begin{enumerate}
		\item  $\huac{\spo^2}=\bigcup_{p\in \spo}p\cdot \R^2$.
		\item An element in $\O^2$ generates the whole space  if and only if it is not cyclic. 
	\end{enumerate}
\end{eg}

	
	\begin{proof}
		Let $x=(x_1,x_2)\notin  \huac{M}$, it is easy to see $x_1,x_2$ are  real linearly indenpendent. Indeed, if $x_1=rx_2$ for some $r\in \R$, then $x=x_2(r,1)$. Note that $(r,1)$ is an associative element of $\O^2$, it  follows from Lemma  \ref{lem:phuaa in huac} that $x\in \huac{M}$, a contradiction. Hence both $x_1,x_2$ are not zero and $(x_1)^{-1}x_2\notin \spr$. Therefore, $$(x_1)^{-1}x=(1,(x_1)^{-1}x_2),\ (x_1)^{-1}x_2\notin \spr.$$ Thus we can choose $p,q\in \spo$, such that $ [p,q,(x_1)^{-1}x_2]\neq0$ (if not, we would have  $[p,q,(x_1)^{-1}x_2]=0$ for any $p,q$, which means $(x_1)^{-1}x_2\in \spr$). However, $$[p,q,(x_1)^{-1}x]=\big([p,q,1],[p,q,(x_1)^{-1}x_2]\big)=\big(0,[p,q,(x_1)^{-1}x_2]\big)\in \left\langle x\right\rangle _\spo,$$ we thus obtain $$\{(0,p)\mid p\in \spo\}\subseteq \left\langle x\right\rangle _\spo.$$
		Similar arguments apply to $(x_2)^{-1}x$, we  get the conclusion as desired.		
	\end{proof}

 An analogous statement holds  for $\conjgt{\O}^2$.
We next consider the case $M=\O\oplus \conjgt{\O}$.

\begin{eg}\label{eg:O+O-}
	$\generat{(1,1)}_\O=\O\oplus \conjgt{\O}$
	\
	
Since	$$e_1(1,1)=(e_1,-e_1),\quad e_2(e_3(1,1))=e_2(e_3,-e_3)=(e_1,e_1),$$
we conclude that both $(e_1,0)$ and $(0,e_1)$ lie in  $\generat{(1,1)}_\O$, which yields that the submodules $\{(0,p)\mid p\in \spo\}$ and $\{(p,0)\mid p\in \spo\}$  both lie in $\generat{(1,1)}_\O$, too. Therefore $\generat{(1,1)}_\O=\O\oplus \conjgt{\O}$.
\end{eg}

In fact, Lemma \ref{lem:phuaa in huac}, Example \ref{eg:(e1,e2)=O}  and Example
\ref{eg:O+O-} are all  specific to a general result that will be proved later.
The following lemma is crucial to set up this result.
\begin{lemma}\label{lem:x in huac(M)=dim =8}
	Let $x$ be any given nonzero element in $M$. Then $$x\in \huac{M}\iff \mathrm{dim}_\spr \generat{x}_\spo=8\iff \generat{x}_\spo\cong \spo \text{ or } \overline{\spo}.$$
\end{lemma}
\begin{proof}
	Let $x\in \huac{M}$, then $\generat{x}_\spo=\spo x$ and hence $\mathrm{dim}_\spr \generat{x}_\spo\leqslant8$. On the other hand, since $\generat{x}_\spo$ is a nonzero  $\spo $-module of finite dimension, thus $\mathrm{dim}_\spr \generat{x}_\spo\geqslant8$, therefore $\mathrm{dim}_\spr \generat{x}_\spo=8$, this means  $ \generat{x}_\spo$ is a simple $\spo$-module and hence $\generat{x}_\spo\cong \spo \text{ or } \overline{\spo}$. Suppose $\generat{x}_\spo\cong \spo \text{ or } \overline{\spo}$. Assume $\generat{x}_\spo\cong \spo $ first. Let $\varphi$ denote an isomorphism: $\varphi:\generat{x}_\spo\rightarrow \spo.$ For any $m\in \generat{x}_\spo$, suppose $\varphi(x)=p,\ \varphi(m)=q$. Then
	$$\varphi(m)=q=qp^{-1}p=qp^{-1}\varphi(x)=\varphi((qp^{-1})x),$$
	according to that $\varphi $ is isomorphism, we  get $m=(qp^{-1})x$, hence $\generat{x}_\spo=\spo x$ which means $x\in \huac{M}$.  If $\generat{x}_\spo\cong \overline{\spo }$,  still let $\varphi$ denote the isomorphism: $\varphi:\generat{x}_\spo\rightarrow \overline{\spo }.$ For any $m\in \generat{x}_\spo$, suppose $\varphi(x)=p,\ \varphi(m)=q$. Then
	$$\varphi(m)=q=\overline{(qp^{-1})}\hat{\cdot}p=\overline{(qp^{-1})}\hat{\cdot}\varphi(x)=\varphi(\overline{(qp^{-1})}x),$$
	then we  get $m=\overline{(qp^{-1})}x$, hence  $x\in \huac{M}$.
\end{proof}

	According to above lemma, we define $$\huacc{+}{M}:=\{x\in \huac{M}\mid \generat{x}_\spo\cong \spo \}\cup\{0\},$$  $$\huacc{-}{M}:=\{x\in \huac{M}\mid \generat{x}_\spo\cong \overline{\spo}\}\cup\{0\}.$$ Therefore
	$\huac{M}=\huacc{+}{M}\cup\huacc{-}{M}.$ We shall show that all the cyclic elements are determined by the associative subset $\huaa{M}$ and the conjugate associative subset $\hua{A}{-}{M}$.

\begin{thm}\label{lem:huac+(M)}
	Let $M$ be a left $\spo$-module, then:
	\begin{enumerate}
		\item $\huacc{+}{M}=\bigcup_{p\in \spo}p\cdot \huaa{M}$;
		\item  $\huacc{-}{M}=\bigcup_{p\in \spo}p\cdot \hua{A}{-}{M}$.
	\end{enumerate}
\end{thm}

\begin{proof}
	We prove a\asertion{1}.  We first show $\bigcup_{p\in \spo}p\cdot \huaa{M}\subseteq\huacc{+}{M}$.  Given any $x\in \huaa{M}$. \bfs\ $x\neq 0$.
	Define a map $\phi:\generat{x}_\spo\rightarrow \spo$ such that $ \phi( px)= p$ for $p\in \spo$.
	This is a homomorphism in $\Hom_\spo(\generat{x}_\spo, \spo)$, since $$\phi(q(px))=\phi((qp)x)=qp=q\phi(px).$$
	Define $\varphi:\spo\rightarrow \generat{x}_\spo$ by $ \varphi(p)= px$. Then
	$$\varphi(pq)=(pq)x=p(qx)=p\varphi(q).$$ Hence $\varphi\in \Hom_\spo( \spo,\generat{x}_\spo)$ and $\phi \varphi=id,\varphi\phi=id$ and thus $\generat{x}_\spo\cong \spo$. This proves $x\in \huacc{+}{M}$. Because $px\in \generat{x}_\spo$ and $x=p^{-1}(px)\in \generat{px}_\spo$ for $p\neq 0$, that is, $\generat{x}_\spo=\generat{px}_\spo$ whenever $p\neq 0$. This implies $\bigcup_{p\in \spo}p\cdot \huaa{M}\subseteq\huacc{+}{M}$.
	On the contary, let $0\neq x\in \huacc{+}{M}$, hence there is an isomorphism $\phi\in \Hom_\spo(\spo,\generat{x}_\spo)$. Suppose $\phi (1)=y\in \generat{x}_\spo$, since $\phi$ is an isomorphism, there is $0\neq r\in \spo$ such that $y=rx$. However in view of Proposition \ref{prop:f(huaaM) in huaaN}, $y=\phi (1)\in \huaa{\generat{x}_\spo}\subseteq \huaa{M}$, thus $x=r^{-1}y\in \bigcup_{p\in \spo}p\cdot \huaa{M}$. This proves a\asertion{1}.

	We prove a\asertion{2}. Easy to show that $\generat{x}_\spo=\spo x$ for $x\in \hua{A}{-}{M}$.
Let $x\in \hua{A}{-}{M}$. \bfs\ $x\neq 0$.
Define a map $\phi:\generat{x}_\spo\rightarrow \overline{\spo}$ such that $ \phi( px)=\overline{ p}$ for $p\in \spo$.
This is a homomorphism in $\Hom_\spo(\generat{x}_\spo, \spo)$, since $$\phi(q(px))=\phi((pq)x)=\overline{ pq}=q\hat{\cdot}\overline{ p}=q\phi(px).$$
Define $\varphi:\overline{\spo}\rightarrow \generat{x}_\spo$ by $ \varphi(p)= \overline{p}x$. Then
$$\varphi(p\hat{\cdot}q)=\varphi(\overline{p}q)=\overline{(\overline{ p}q)}x={ p}(\overline{q}x)=p\varphi(q).$$ Hence $\varphi\in \Hom_\spo( \spo,\generat{x}_\spo)$ and $\phi \varphi=id,\varphi\phi=id$ and thus $\generat{x}_\spo\cong \overline{\spo}$. This proves $x\in \huacc{-}{M}$. Hence we conclude from the fact $\generat{x}_\spo=\generat{px}_\spo$ that  $\bigcup_{p\in \spo}p\cdot \hua{A}{-}{M}\subseteq\huacc{-}{M}$.
On the contary, let $0\neq x\in \huacc{-}{M}$, hence there is an isomorphism $\phi\in \Hom_\spo(\overline{\spo},\generat{x}_\spo)$. Suppose $\phi (1)=y\in \generat{x}_\spo$, since $\phi$ is an isomorphism, there is $0\neq r\in \spo$ such that $y=rx$. Note that for any $p,q\in \spo$,
$$(pq)y=(pq)\phi(1)=\phi((pq)\hat{\cdot}1)=\phi(q\hat{\cdot} \overline{p})=q\phi(p\hat{\cdot}1)=q(py).$$
This shows $y\in \hua{A}{-}{M}$ and
 thus $x=r^{-1}y\in \bigcup_{p\in \spo}p\cdot \hua{A}{-}{M}$. This proves a\asertion{2}.

\end{proof}
In view of Theorem \ref{thm:strc of left O-mod}, we conclude an important consequence of the above Theorem:
\begin{cor}\label{cor:M=Span C(M)}
	For each left $\O$-module $M$, we have $$M=\text{Span}_\R\hua{C}{+}{M}\oplus\text{Span}_\R\hua{C}{-}{M}=\text{Span}_\R\huac{M}.$$
\end{cor}
\begin{rem}
	The corollary shows that for any element $m\in M$, there exist some real linear independent cyclic elements $m_i\in \huac{M}$, $i=1,\ldots,n$, such that 
	$$m=\sum_{i=1}^n m_i.$$
	However, this decomposition is not unique.  For example, let $M=\O^3$ and $m=(1,1+e_1,e_1)$. Then choose $m_1=(1,0,0),m_2=(0,1+e_1,0),m_3=(0,0,e_1)$ in $\huac{M}$, it clearly holds $m=\sum_{i=1}^3 m_i$ and they are real linear independent.  On the other hand, we can choose 
	$m_1'=(1,1,0), m_2'=e_1(0,1,1)$ in $\huac{M}$, this is also a decomposition of $m$. 
	It is worth noticing the length of a decomposition of $m$ must be no less than $2$ and no more than $3$. Loosely speaking, the  length of the decomposition of one element reflects the size of the submodule generated by it. The accurate relation deserves further study.

	For any element $m$ in a  left $\O$-module $M$, it follows from  Corollary \ref{cor:M=Span C(M)} that there exist $m^\pm\in \text{Span}_\R\hua{C}{\pm}{M}$ such that $m=m^++m^-$.  We can decomposite $m^{\pm}$ into  a combination of real linearly independent cyclic elements. Denote by $l_m^{\pm}$ the minimal length of the decompositions of $m^{\pm}$. \textbf{We conjecture that,} $$\generat{m}_\O\cong \O^{l_m^+}\oplus\O^{l_m^-}.$$
	It holds true in above example. In fact, in a similar manner as in Example \ref{eg:(e1,e2)=O}, we can prove  $$\generat{(1,1+e_1,e_1)}_\O=\O\cdot (1,1,0)\oplus \O\cdot  e_1(0,1,1)\cong \O^2.$$
	If the conjecture is right, then the structure of the submodule generated by one element is completely clear.
\end{rem}
\bibliographystyle{plain}

\bibliography{refenrence}

\begin{thebibliography}{10}

\bibitem{Alperin2012gtm162}
J.~L. Alperin and Rowen~B. Bell.
\newblock {\em Groups and representations}, volume 162 of {\em Graduate Texts
  in Mathematics}.
\newblock Springer-Verlag, New York, 1995.

\bibitem{Atiyah1964Cliffordmodules}
M.~F. Atiyah, R.~Bott, and A.~Shapiro.
\newblock Clifford modules.
\newblock {\em Topology}, 3(suppl, suppl. 1):3--38, 1964.

\bibitem{baez2002octonions}
John~C. Baez.
\newblock The octonions.
\newblock {\em Bull. Amer. Math. Soc. (N.S.)}, 39(2):145--205, 2002.

\bibitem{colombo2011noncomfunctcalculus}
Fabrizio Colombo, Irene Sabadini, and Daniele~C. Struppa.
\newblock {\em Noncommutative functional calculus}, volume 289 of {\em Progress
  in Mathematics}.
\newblock Birkh\"{a}user/Springer Basel AG, Basel, 2011.
\newblock Theory and applications of slice hyperholomorphic functions.

\bibitem{eilenberg1948extensions}
Samuel Eilenberg.
\newblock Extensions of general algebras.
\newblock {\em Ann. Soc. Polon. Math.}, 21:125--134, 1948.

\bibitem{ghiloni2013slicefct}
Riccardo Ghiloni, Valter Moretti, and Alessandro Perotti.
\newblock Continuous slice functional calculus in quaternionic {H}ilbert
  spaces.
\newblock {\em Rev. Math. Phys.}, 25(4):1350006, 83, 2013.

\bibitem{gilbert1991clifford}
John~E. Gilbert and Margaret A.~M. Murray.
\newblock {\em Clifford algebras and {D}irac operators in harmonic analysis},
  volume~26 of {\em Cambridge Studies in Advanced Mathematics}.
\newblock Cambridge University Press, Cambridge, 1991.

\bibitem{goldstine1964hilbert}
H.~H. Goldstine and L.~P. Horwitz.
\newblock Hilbert space with non-associative scalars. {I}.
\newblock {\em Math. Ann.}, 154:1--27, 1964.

\bibitem{harvey1990spinors}
F.~Reese Harvey.
\newblock {\em Spinors and calibrations}, volume~9 of {\em Perspectives in
  Mathematics}.
\newblock Academic Press, Inc., Boston, MA, 1990.

\bibitem{horwitz1993QHilbertmod}
L.~P. Horwitz and A.~Razon.
\newblock Tensor product of quaternion {H}ilbert modules.
\newblock In {\em Classical and quantum systems ({G}oslar, 1991)}, pages
  266--268. World Sci. Publ., River Edge, NJ, 1993.

\bibitem{jacobson1954structure}
N.~Jacobson.
\newblock Structure of alternative and {J}ordan bimodules.
\newblock {\em Osaka Math. J.}, 6:1--71, 1954.

\bibitem{ludkovsky2007algebras}
S.~V. Ludkovsky.
\newblock Algebras of operators in {B}anach spaces over the quaternion skew
  field and the octonion algebra.
\newblock {\em Sovrem. Mat. Prilozh.}, (35):98--162, 2005.

\bibitem{ludkovsky2007Spectral}
S.~V. Ludkovsky and W.~Spr\"{o}ssig.
\newblock Spectral representations of operators in {H}ilbert spaces over
  quaternions and octonions.
\newblock {\em Complex Var. Elliptic Equ.}, 57(12):1301--1324, 2012.

\bibitem{Sommen2012spinor}
Alan McIntosh.
\newblock Book {R}eview: {C}lifford algebra and spinor-valued functions, a
  function theory for the {D}irac operator.
\newblock {\em Bull. Amer. Math. Soc. (N.S.)}, 32(3):344--348, 1995.

\bibitem{ng2007quaternionic}
Chi-Keung Ng.
\newblock On quaternionic functional analysis.
\newblock {\em Math. Proc. Cambridge Philos. Soc.}, 143(2):391--406, 2007.

\bibitem{razon1992Uniqueness}
A.~Razon and L.~P. Horwitz.
\newblock Uniqueness of the scalar product in the tensor product of quaternion
  {H}ilbert modules.
\newblock {\em J. Math. Phys.}, 33(9):3098--3104, 1992.

\bibitem{razon1991projection}
Aharon Razon and L.~P. Horwitz.
\newblock Projection operators and states in the tensor product of quaternion
  {H}ilbert modules.
\newblock {\em Acta Appl. Math.}, 24(2):179--194, 1991.

\bibitem{rotman2017advancedalg}
Joseph~J. Rotman.
\newblock {\em Advanced modern algebra. {P}art 2}, volume 180 of {\em Graduate
  Studies in Mathematics}.
\newblock American Mathematical Society, Providence, RI, third edition, 2017.
\newblock With a foreword by Bruce Reznick.

\bibitem{schafer2017introduction}
Richard~D. Schafer.
\newblock {\em An introduction to nonassociative algebras}.
\newblock Dover Publications, Inc., New York, 1995.
\newblock Corrected reprint of the 1966 original.

\bibitem{soffer1983quaternion}
A.~Soffer and L.~P. Horwitz.
\newblock {$B^{\ast} $}-algebra representations in a quaternionic {H}ilbert
  module.
\newblock {\em J. Math. Phys.}, 24(12):2780--2782, 1983.

\bibitem{viswanath1971normal}
K.~Viswanath.
\newblock Normal operations on quaternionic {H}ilbert spaces.
\newblock {\em Trans. Amer. Math. Soc.}, 162:337--350, 1971.

\bibitem{wang2014octonion}
Haiyan Wang and Guangbin Ren.
\newblock Octonion analysis of several variables.
\newblock {\em Commun. Math. Stat.}, 2(2):163--185, 2014.

\end{thebibliography}
\end{document}